\documentclass[12pt,reqno]{amsart}
\usepackage[margin=2in]{geometry}
\usepackage[utf8]{inputenc}
\usepackage[backend=bibtex,style=alphabetic,maxbibnames=99,minbibnames=99]{biblatex}
\usepackage{graphicx}
\usepackage{bookmark}
\usepackage{orcidlink}

\usepackage{comment}
\graphicspath{ {./images/} }
\usepackage{amsmath, amsthm, amssymb, mathrsfs, mathtools, mathdots}
\usepackage{tikz-cd}
\usepackage[framemethod=TikZ]{mdframed}

\usepackage{thmtools}

\numberwithin{equation}{section}

\usepackage{subfiles}
\usepackage{xcolor}
\usepackage{soul}
\usepackage{graphicx}
\usepackage{marginnote}
\usepackage{fancyhdr}
\usepackage{hyperref}
\usepackage{mathpazo}
\hypersetup{
    colorlinks=true,
    linkcolor=blue,
    filecolor=magenta,      
    urlcolor=cyan,
    pdftitle={Classification of Finite Groups With Equal Left and Right Quotient Sets},
    pdfpagemode=FullScreen,
    citecolor=blue,
}

\makeatletter
\def\thm@space@setup{\thm@preskip=1em
\thm@postskip=1em}
\makeatother

\theoremstyle{plain}
\newtheorem{theorem}{Theorem}[section]
\newtheorem{lemma}[theorem]{Lemma}
\newtheorem{corollary}[theorem]{Corollary}

\newtheorem{proposition}[theorem]{Proposition}
\newtheorem{question}[theorem]{Question}
\theoremstyle{definition}
\newtheorem{definition}[theorem]{Definition}
\theoremstyle{remark}
\newtheorem{remark}[theorem]{Remark}
\theoremstyle{plain}

\newcommand{\pv}[1]{\left\langle #1 \right\rangle} 
\newcommand{\ps}[1]{\left\{ #1 \right\}} 
\newcommand{\inv}{^{-1}}
\newcommand{\Z}{\mathbb{Z}}

\newcommand{\Aut}{\mathrm{Aut}}

\title{Classification of Finite Groups With Equal Left and Right Quotient Sets}
\author{Haran Mouli\,\orcidlink{0009-0000-5891-6618}}\email{hmouli@wisc.edu}
\address{Department of Mathematics, University of Wisconsin, Madison, WI, 53706}

\author{Pramana Saldin\,\orcidlink{0009-0006-7395-8342}}\email{saldin@wisc.edu}
\address{Department of Mathematics, University of Wisconsin, Madison, WI, 53706}

\date{\today}

\begin{document}

\begin{abstract}
    In this paper, we classify all finite groups $G$ which have the following property: for all subsets $A \subseteq G$, we have $|AA^{-1}| = |A^{-1}A|$. This question is motivated by the problem in additive combinatorics of More Sums Than Difference sets and answers several questions posed in \cite{duvivier2025}. 
\end{abstract}
\maketitle 

\tableofcontents

\section{Introduction}

{\color{red}
The authors have been made aware of existing literature which independently proved \autoref{thm:classification-of-balanced-groups} \cite{Herzog-Kaplan-Longobardi-Maj}. 
We are grateful to Liubomir Chiriac for showing us this result after reading the paper posing this problem from the second author \cite{duvivier2025}.
}

Let \(G\) be a finite group. We say that \(G\) is \textbf{balanced} 
if for all \(A \subseteq G\), the \textbf{left quotient set} and the \textbf{right quotient set}, defined as 
\begin{align*}
    AA\inv &\coloneqq \ps{a_ia_j\inv : a_i,a_j\in A}, \text{ and}\\
    A\inv A &\coloneqq \ps{a_i\inv a_j : a_i,a_j\in A},
\end{align*}
respectively, have equal cardinality. 
Otherwise, we say $G$ is \textbf{imbalanced}. 

This question of whether or not a group is balanced 
is an extension of problems regarding 
More Sums Than Differences (MSTD) sets, which asks about finite subsets \(A\) 
of groups where (using additive notation) \(|A+A| > |A-A|\), \(|A+A| < |A-A|\), and \(|A+A| = |A-A|\). 
For example, these sets have been studied in finite abelian groups, \(\Z\), \(\Z^n\), and \(D_\infty = C_2 \ltimes \Z\) \cite{zhao2010counting,ascoli2022sum,martin2006setssumsdifferences,kim2020constructionsgeneralizedmstdsets}. 

In a recent paper coauthored by the second author \cite{duvivier2025}, the 
following two questions were posed: is there an infinite family of non-abelian 
balanced groups, and what are the necessary and sufficient conditions 
for a group \(G\) to be balanced? 

We resolve both questions by classifying all balanced groups.
\begin{theorem}\label{thm:classification-of-balanced-groups}
    The balanced groups are precisely: 
    \begin{enumerate}
        \item Finite abelian groups,
        \item $Q_8 \times (C_2)^n$ for $n \ge 0$,
        \item \(D_6\), \(D_8\), \(D_{10}\), \(Q_{12}\), 
    \(Q_{16}\), \(C_4 \ltimes C_4\), and \(Q_{20}\).
    \end{enumerate}
\end{theorem}

In Section \ref{sec:The-Hamil-2-nian-Groups}, we prove that Hamiltonian $2$-groups $Q_8 \times (C_2)^n$ are balanced, which give an example of an infinite family of non-abelian balanced groups. In Section \ref{sec:Weakly-Balanced-Groups}, we define the notion of a weakly balanced group, where one checks that $|AA^{-1}| = |A^{-1}A|$ for certain special subsets $A \subseteq G$ suggested in \cite{tao} for which the equality of cardinalities is not generally expected to hold. The condition of being weakly balanced is easier to study and provides a first filter on which finite groups are balanced. The main result we use is \autoref{cor:three cases of weakly balanced groups}, which shows that every weakly balanced group must be one of three types. Finally, in Section \ref{sec:Classification-of-Balanced-Groups}, we analyze these three cases separately and classify all balanced groups using a combination of group-theoretic arguments and some computation in SageMath.

\begin{remark}
    For this paper, we use the following notation for groups: 
    \begin{itemize}
        \item \(C_n = \pv{t \mid t^n = 1}\) is the cyclic group of order \(n\). 
        \item \(D_{2n}= \pv{r,s \mid r^n = s^2 = 1, sr = r\inv s}\) is the dihedral group of order \(2n\). 
        \item \(Q_{4n}\) is the dicyclic group of order \(4n\). \item In particular, \(Q_8=\pv{i,j,k \mid i^2=j^2=k^2=ijk=-1}\) is the quaternion group. 
    \end{itemize}
\end{remark}

\subsection*{Acknowledgements} The second author would like to thank the SMALL REU program, and in particular the research group focused on MSTD sets, for raising questions related to balanced groups. This group included undergraduate students Julian Duvivier, Xiaoyao Huang, Ava Kennon, Say-Yeon Kwon, Arman Rysmakhanov, and Ren Watson, along with the second author. The project was advised by Steven J. Miller, to whom the second author is sincerely grateful for his guidance and support. The second author is also thankful to the other group members for their collaboration and insight.

\subsection*{Disclosure Statement}
The authors report there are no competing interests to declare.

\section{The Hamil-2-nian Groups}\label{sec:The-Hamil-2-nian-Groups}
The following class of groups will be useful to consider for our classification.
\begin{definition}
    A group \(G\) is \textbf{Dedekind} if every subgroup of 
    \(G\) is normal.
\end{definition}
Note that every abelian group is Dedekind. The smallest example of a non-abelian Dedekind group is $Q_8$.
\begin{definition}
    A group \(G\) is \textbf{Hamiltonian} if it is Dedekind and non-abelian.
\end{definition}

Importantly, the structure of Hamiltonian groups is well-understood due to 
a classical result by Dedekind and Baer. 
\begin{theorem}[\cite{baer}]\label{thm:baer-structure}
    A group is Hamiltonian if and only if it is isomorphic to \(Q_8\times P\times Q\), where \(P\) is an elementary abelian \(2\)-group (the direct product of copies of \(C_2\)), and \(Q\) is a torsion abelian group with all elements of odd order. 
\end{theorem}

We now consider finite Hamiltonian \(2\)-groups (Hamil-\(2\)-nian groups), which by \autoref{thm:baer-structure} are of the form \(Q_8 \times (C_2)^n\). 

\begin{proposition}\label{prop:Hamil-2-nian}
    \(Q_8 \times (C_2)^n\) is balanced for all \(n\ge 0\). 
\end{proposition}
\begin{proof}
    Let \(A \subseteq Q_8 \times (C_2)^n\). We claim that in fact, 
    \[
        AA\inv = A\inv A. 
    \]
    We say that \(x,y\in Q_8\times (C_2)^n\) \textbf{anti-commute} if 
    \(xy = -yx\). Every pair of elements of \(Q_8\) either commute or anti-commute, 
    so the same is true for \(Q_8 \times (C_2)^n\). In order to show that $AA^{-1} = A^{-1}A$, it suffices to prove that for all $x, y \in A$:
    \[
        \ps{xy\inv, yx\inv} = \ps{x\inv y, y\inv x}.
    \]

    \begin{itemize}
        \item \underline{Case 1: \(x\) and \(y\) anti-commute.} 
        In this case:
        \[
            \ps{xy\inv, yx\inv} = \ps{-x\inv y, -y\inv x}.
        \]
        Since \(x\) and \(y\) anti-commute, the elements $x\inv y$ and $y\inv x$ are not central. Since the inverse of any non-central element $z$ of $Q_8 \times (C_2)^n$ is equal to $-z$, it follows that:
        \[
            \ps{-x\inv y, -y\inv x} = \ps{(x\inv y)\inv, (y\inv x)\inv} = \ps{y\inv x, x\inv y}
        \]
        from which the desired equality follows.

        \item \underline{Case 2: \(x\) and \(y\) commute.}
        Hence \(x\inv\) commutes with \(y\) and \(y\inv\) commutes with 
        \(x\). It immediately follows that 
        \(\ps{xy\inv, yx\inv} = \ps{x\inv y, y\inv x}\). \qedhere
    \end{itemize}
\end{proof}

Thus, there exists an infinite family of distinct non-abelian balanced groups. 

\section{Weakly Balanced Groups}\label{sec:Weakly-Balanced-Groups}

In \cite{tao}, Tao makes the following observation: let $G$ be a finite group and $H \le G$ be non-normal. If $g \not\in N_G(H)$ and $A = H \cup gH$, then $|AA^{-1}| = \mathcal{O}(|H|)$, but $|A^{-1}A|$ can generally be bigger. The following proposition make this precise:

\begin{proposition}\label{prop:balanced-equivalent-coset-condition}
    Let $H \le G$ be non-normal. Let $g \in G - N_G(H)$ and $A = H \cup gH$. Then, the following are equivalent:
    \begin{enumerate}
        \item $|AA^{-1}| = |A^{-1}A|$
        \item $|H \cap gHg^{-1}| = |gH \cap Hg^{-1}| = \tfrac{1}{2}|H|$.
    \end{enumerate}
\end{proposition}

\begin{proof}
    First, we prove that $(1)$ implies $(2)$. We have:
    \begin{enumerate}
        \item $AA^{-1} = H \cup gH \cup Hg^{-1} \cup gHg^{-1}$.
        \item $A^{-1}A = H \cup HgH \cup Hg^{-1}H$
    \end{enumerate}
    On one hand, $|AA^{-1}| < |H| + |gH| + |Hg^{-1}| + |gHg^{-1}| = 4|H|$, where the inequality is strict since $1 \in H \cap gHg^{-1}$. On the other hand, $A^{-1}A$ contains the double cosets $HgH$ and $Hg^{-1}H$, both unequal to $H$ since $g \not\in H$. The sizes of these double cosets are both equal to $\tfrac{|H|^2}{|H \cap gHg^{-1}|}$ by the orbit-stabilizer theorem. Since $g \not\in N_G(H)$, this is a multiple of $|H|$ and at least $2|H|$. 
    \begin{itemize}
        \item If $HgH \neq Hg^{-1}H$, then $H$, $HgH$, and $Hg^{-1}H$ would be pairwise disjoint double cosets, so $|A^{-1}A| = 5|H| > |AA^{-1}|$, which is a contradiction. Thus, we must have $HgH = Hg^{-1}H$.
        \item If $|HgH| \ge 3|H|$, then $|A^{-1}A| \ge |H| + 3|H| > |AA^{-1}|$, which is also a contradiction. Thus, $|HgH| = |Hg^{-1}H| = 2|H|$ and $|H \cap gHg^{-1}| = \tfrac{1}{2}|H|$.
    \end{itemize} 
    
    We now have the precise equality $|A^{-1}A| = |H| + 2|H| = 3|H|$, so we must also have $|AA^{-1}| = 3|H|$. Observe that both $gH$ and $Hg^{-1}$ are both non-trivial left or right cosets of both $H$ and $gHg^{-1}$, so:
    \[|AA^{-1}| = |H \cup gHg^{-1}| + |gH \cup Hg^{-1}| = 4|H| - |H \cap gHg^{-1}| - |gH \cap Hg^{-1}|\]
    is equal to $3|H|$. We have already calculated that $|H \cap gHg^{-1}| = \tfrac{1}{2}|H|$, so it follows that $|gH \cap Hg^{-1}| = \tfrac{1}{2}|H|$, showing $(1)$ implies $(2)$.

    For the other direction, $|H \cap gHg^{-1}| = |gH \cap Hg^{-1}| = \tfrac{1}{2}|H|$ gives $|AA^{-1}| = 3|H|$ as in the above calculation. It also implies that $|HgH| = |Hg^{-1}H| = 2|H|$. We have $HgH = Hg^{-1}H$ since $gH \cap Hg^{-1} \neq \emptyset$. Thus, $|A^{-1}A| = 3|H|$ as well, proving that $(2)$ implies $(1)$.
\end{proof}

\begin{definition}
    We say that a finite group $G$ is \textbf{weakly balanced} if $|AA^{-1}| = |A^{-1}A|$ when $A = H \cup gH$ for any non-normal subgroup $H$ and $g \not\in N_G(H)$.
\end{definition}

It is clear that every balanced group is weakly balanced.

\begin{proposition}\label{prop:semidirect product of 2-group and odd abelian group}
    If $G$ is weakly balanced, then $G = P \ltimes Q$ is a semidirect product, where $P$ is a $2$-group and $Q$ is an odd abelian group.
\end{proposition}

\begin{proof}
    Let $H$ be any odd subgroup of $G$. By \autoref{prop:balanced-equivalent-coset-condition}, $H$ must be normal since we cannot have $[H: H \cap gHg^{-1}] = 2$ for any $g \in G$. Furthermore, every subgroup of $H$ is normal in $G$, and consequently $H$. It follows that $H$ is Dedekind, so by \autoref{thm:baer-structure}, it must be abelian.

    Now, the Sylow $p$-subgroups of $G$ for odd primes $p$ are normal and abelian, so the subgroup $Q$ they generate is isomorphic to their direct product, which must be odd abelian. Furthermore, $Q$ is a normal complement to $P$, so $G = P \ltimes Q$.
\end{proof}

For the remainder of this paper, we shall use $P$ and $Q$ as in the statement of \autoref{prop:semidirect product of 2-group and odd abelian group}. Next, we study the possibilities for the map $\Phi:P \to \Aut(Q)$ associated to the conjugation action of $P$ on $Q$. Recall that the map is defined by $\Phi: x \mapsto \Phi_x$, where $\Phi_x(y) := xyx^{-1}$ for all $x \in P$ and $y \in Q$.

\begin{lemma}\label{lem:squares in H with probability 1/2}
    If $G$ is weakly balanced and $H$ is a non-normal subgroup, then exactly half the elements of any (left) $H$-coset outside $N_G(H)$ have squares in $H$.
\end{lemma}

\begin{proof}
    Let $gH$ be any left coset outside $N_G(H)$. For any $x \in gH$, we have $x^2 \in H$ if and only if $x \in Hx^{-1}$, which is the same as $Hg^{-1}$. Thus, the elements of $gH$ with squares in $H$ are precisely the elements in $gH \cap Hg^{-1}$. By \autoref{prop:balanced-equivalent-coset-condition}, this is exactly half the elements of $gH$.
\end{proof}

\begin{lemma}
    The map $\Phi: P \to \Aut(Q)$ has image contained in $\{\pm 1\}$, where $1$ and $-1$ denote the identity and inverse automorphisms of $Q$.
\end{lemma}

\begin{proof}
    For any $y \in N_G(P) \cap Q$ and $x \in P$, $yxy^{-1} = y\Phi_x(y)^{-1}x \in P$, so we must have $y = \Phi_x(y)$. Since $x \in P$ was arbitrary, $y$ must commute with $P$ (and $Q$ since it is abelian), implying that $y \in Z(G)$. In other words, $N_G(P) \cap Q = Z(G) \cap Q$.

    Now, let $y \in Q$ be any element. If $y \in Z(G) \cap Q$, then the orbit of $y$ under the conjugation action of $P$ is equal to $\{y\}$. Otherwise, we have $y \not\in N_G(P)$. In this case, for any $x \in P$, we have $(yx)^2 = y\Phi_x(y)x^2$ in $P$ if and only if $\Phi_x(y) = y^{-1}$. By \autoref{lem:squares in H with probability 1/2}, this happens for half the elements $x \in P$, forcing the orbit of $y$ under the conjugation action of $P$ to be $\{y, y^{-1}\}$.

    Finally, assume for the sake of contradiction that there exist $x \in P$ and $y_1, y_2 \in Q$ such that $\Phi_x(y_1) = y_1$ and $\Phi_x(y_2) = y_2^{-1}$. In this case, $\Phi_x(y_1y_2) = y_1y_2^{-1}$ is neither equal to $y_1y_2$, nor its inverse, which is a contradiction. Thus, $\Phi_x$ must be equal to $\pm 1$ for all $x \in P$, which proves the proposition.
\end{proof}

\begin{definition}
    We say that $G = P \ltimes Q$ is a \textbf{sign semidirect product} if $Q$ is non-trivial and the map $\Phi: P \to \Aut(Q)$ has image equal to $\{\pm 1\}$.
\end{definition}

\begin{corollary}\label{cor:three cases of weakly balanced groups}
    Let $G$ be a non-abelian weakly balanced group. Let $G = P \ltimes Q$ as in \autoref{prop:semidirect product of 2-group and odd abelian group}. Exactly one of the following is true:
    \begin{enumerate}
        \item $G = P \times Q$ is a direct product, where $Q$ is non-trivial.
        \item $G = P \ltimes Q$ is a sign semidirect product.
        \item $G = P$ is a non-abelian $2$-group.
    \end{enumerate}
\end{corollary}

\section{Classification of Balanced Groups}\label{sec:Classification-of-Balanced-Groups}

Note that abelian groups are trivially balanced, so if suffices to 
focus on the non-abelian case. 
We classify all balanced groups by using \autoref{cor:three cases of weakly balanced groups} to divide into the three cases.
\autoref{thm:classification-of-balanced-groups} will follow from proving that the non-abelian balanced examples in these cases are as follows: 
\begin{enumerate}
    \item Direct Product: No examples.
    \item Sign Semidirect Product: $D_6, D_{10}, Q_{12}, Q_{20}$.
    \item Non-abelian $2$-groups: $D_8$, $Q_{16}$, $C_4 \ltimes C_4$, and $Q_8 \times (C_2)^n$ for $n \ge 0$.
\end{enumerate}
The following observation will be crucial in our classification:

\begin{lemma}\label{lem:subquotients}
    Subquotients of (weakly) balanced groups are (weakly) balanced.
\end{lemma}

\begin{proof}
    It is clear that subgroups of (weakly) balanced groups are weakly balanced. Similarly, for any quotient $G \twoheadrightarrow \tilde{G}$, the claim $|\tilde{A}\tilde{A}^{-1}| = |\tilde{A}^{-1}\tilde{A}|$ for some $\tilde{A} \subseteq G$ is equivalent to $|AA^{-1}| = |A^{-1}A|$ for its preimage $A \subseteq G$. Furthermore, if $\tilde{A}$ has the form $\tilde{H} \cup \tilde{g}\tilde{H}$, then $A = H \cup gH$, where $H$ is the preimage of $\tilde{H}$ and $g$ is any representative of $\tilde{g}$ in $G$. Thus, if $G$ is (weakly) balanced, then so is $\tilde{G}$. The lemma follows from combining the result for subgroups and quotients.
\end{proof}

\subsection{Direct Products}

Let $G = P \times Q$, where $P$ is a non-abelian $2$-group and $Q$ is a non-trivial odd abelian group.

\begin{lemma}\label{lem:Direct products are weakly balanced iff Hamiltonian}
    $G$ is weakly balanced if and only if $P$ (equivalently $G$) is Hamiltonian. 
\end{lemma}

\begin{proof}
    It is vacuously true that Hamiltonian groups are weakly balanced. We prove the forward direction using the contrapositive. If $P$ is not Hamiltonian, we can find a non-normal subgroup $H \le P$. 
    
    Pick any non-identity element $y \in Q$. We know that $y \in N_G(H)$ since $y$ is in $Z(G)$, so we can divide the complement of $N_G(H)$ into cosets of $\langle y \rangle$. For any coset $x \langle y \rangle$, there can be at most one element whose square is in $H$. Indeed, if $(xy^a)^2 = x^2y^{2a}$ and $(xy^b)^2 = x^2y^{2b}$ are both in $H$, then so is $y^{2a-2b}$, which forces $y^{2a} = y^{2b}$, and hence $y^a = y^b$. However, this would mean that the proportion of elements outside $N_G(H)$ with squares in $H$ is at most $\tfrac{1}{|y|} < \tfrac{1}{2}$, which contradicts \autoref{lem:squares in H with probability 1/2}.
\end{proof}

\begin{lemma}\label{lem:Q8xCn is imbalanced}
    \(Q_8\times C_n\) is imbalanced for \(n \ge 3\).
\end{lemma}
\begin{proof}
    Consider the subset $A = \ps{(1,1),(i,1),(j,1),(-1,t),(-i,t\inv), (k,t\inv)}$.
    \begin{enumerate}
        \item For \(n=3\), \(|AA\inv| = 17 \neq 19 = |A\inv A|\).
        \item For \(n \ge 4\), \(|AA\inv| = 21 \neq 23 = |A\inv A|\).
    \end{enumerate}
    Thus, for $n \ge 3$, the group $Q_8 \times C_n$ is imbalanced.
\end{proof}

\begin{corollary}
    There are no balanced groups $G$ of type $(1)$ in \autoref{cor:three cases of weakly balanced groups}.
\end{corollary}

\begin{proof}
    Assume for the sake of contradiction that $G$ in case $(1)$ is balanced. By \autoref{lem:Direct products are weakly balanced iff Hamiltonian}, $P$ must be Hamiltonian. By \autoref{thm:baer-structure}, $G$ has a subgroup of the form $Q_8 \times C_n$ must be abelian for some odd $n > 1$. However, such a subgroup would be imbalanced by \autoref{lem:Q8xCn is imbalanced}, which contradicts \autoref{lem:subquotients}.
\end{proof}

\subsection{Sign Semidirect Products}

Next, we consider the case where $G = P \ltimes Q$ is a balanced sign semidirect product. Let $\Phi: P \to \Aut(Q)$ be the conjugation action map with image equal to $\{\pm 1\}$ and let $K = \ker \Phi$.

\begin{lemma}\label{lem:dihedral-imbalanced}
    \(D_{2n}\) is imbalanced for \(n \ge 6\).
\end{lemma}
\begin{proof}
    Consider the subset \(A = \ps{1, r, r^{-2}, s, sr^{-1}, sr^2}\).
    \begin{enumerate}
        \item For $n = 6$, \(|AA\inv| = 11 \neq 12 = |A\inv A|\).
        \item For $n = 7$, \(|AA\inv| = 13 \neq 14 = |A\inv A|\).
        \item For $n \ge 8$, \(|AA\inv| = 13 \neq 14 = |A\inv A|\).
    \end{enumerate}
    Thus, for $n \ge 6$, the group $D_{2n}$ is imbalanced.
\end{proof}

\begin{lemma}\label{lem:sign-semidirect-imbalanced-C2:(Cp)^2}
    The sign semidirect products $C_2 \ltimes (C_3)^2$ and 
    \(C_2 \ltimes (C_5)^2\) are imbalanced. 
\end{lemma}

\begin{proof}
    In $C_2 \ltimes (C_n)^2 = \langle t, u, v \mid t^2 = u^n = v^n = 1, tu = u^{-1}t, tv = v^{-1}t\rangle$, consider the subset $A = \{u^2v, tu^2, uv, v^2, tv, tuv\}$. One can check that for both $n = 3$ and $n=5$, $|AA^{-1}| = 13 \neq 14 = |A^{-1}A|$. 
\end{proof}

\begin{proposition}\label{prop:sign-semidirect-Q-C3-or-C5}
    If $G = P \ltimes Q$ is a balanced sign semidirect product, then $Q \in \{C_3, C_5\}$.
\end{proposition}

\begin{proof}
    Let \(G = P\ltimes Q\) be a balanced sign semidirect product. The sign semidirect product \(\tilde{G} = C_2 \ltimes Q\) is a quotient of $G$, where the kernel of $G \twoheadrightarrow \tilde{G}$ is given by $K \times Q$, and $\tilde{G}$ is balanced by \autoref{lem:subquotients}. Now, suppose $Q$ has a non-identity element of order $m$. Then, $\tilde{G}$ has a subgroup of the form \(C_2 \ltimes C_m\), which is also balanced by \autoref{lem:subquotients}. But \(C_2\ltimes C_m = D_{2m}\), so by \autoref{lem:dihedral-imbalanced}, we must have \(m\in \ps{3,5}\). 
    
    Furthermore, $Q$ cannot have both elements of order $3$ and $5$, since it would then have elements of order $15$. This means that $Q = (C_3)^n$ or $Q = (C_5)^n$. We cannot have $n \ge 2$ since we would then have a subgroup of the form $C_2 \ltimes (C_3)^2$ or $C_2 \ltimes (C_5)^2$, which must be imbalanced by \autoref{lem:sign-semidirect-imbalanced-C2:(Cp)^2}, so $Q = C_3$ or $Q = C_5$.
\end{proof}

\begin{proposition}\label{prop:sign-semidirect-P-cyclic}
    If $G = P \ltimes Q$ is a balanced sign semidirect product, then $P$ is cyclic.
\end{proposition}

\begin{proof}
    Assume for the sake of contradiction that $P$ is non-cyclic. We may further assume that $P$ has minimal order among non-cyclic examples. Since $K \le P$ is a non-trivial normal subgroup, $L = K \cap Z(P)$ is also non-trivial, so we can find an element $x \in L$ of order $2$. Let $\tilde{P} = P/\langle x \rangle$; the map $\Phi$ factors through $\tilde{P}$, so the sign semidirect product $\tilde{G} = \tilde{P} \ltimes Q$ (given by the induced map $\tilde{P} \to \Aut(Q)$) is a quotient of $G$. By \autoref{lem:subquotients}, we know that $\tilde{G}$ is also balanced.
    
    By the minimality of $P$, we have $\tilde{P} = C_{2^n}$ for some positive integer $n$. Since $P/Z(P)$ is a quotient of $\tilde{P}$, it must also be cyclic, implying that $P$ is abelian. This means that the only non-cyclic choice for $P$ is $C_2 \times C_{2^n}$. Now, consider the subgroup $H$ of index $2$ inside $C_{2^n}$: its elements are precisely the squares in $P$, and hence, it is a normal subgroup contained in $K$. This means that the sign semidirect product $P/H \ltimes Q = (C_2 \times C_2) \ltimes Q$ must be balanced by \autoref{lem:subquotients}.

    However, we know by \autoref{prop:sign-semidirect-Q-C3-or-C5} that $Q \in \{C_3, C_5\}$, and the sign semidirect products $(C_2 \times C_2) \ltimes C_3$ and $(C_2 \times C_2) \ltimes C_5$ are simply the groups $D_{12} = C_2 \times D_6$ and $D_{20} = C_2 \times D_{10}$ respectively. These are imbalanced by \autoref{lem:dihedral-imbalanced}, so we have the required contradiction.
\end{proof}

\begin{lemma}\label{lem:sign-semidirect-imbalanced-C8:Cp}
    The sign semidirect products $C_8 \ltimes C_3$ and 
    \(C_8 \ltimes C_5\) are imbalanced. 
\end{lemma}

\begin{proof}
    In the group $C_8 \ltimes C_n = \langle t, w \mid t^8 = w^n = 1, tw = w^{-1}t\rangle$, consider the subset $A = \{1, t^3, tw, t^6w, t^2w^2, t^4w^2\}$.
    \begin{enumerate}
        \item For $n = 3$, $|AA^{-1}| = 17 \neq 19=  |A^{-1}A|$.
        \item For $n = 5$, $|AA^{-1}| = 23 \neq 25= |A^{-1}A|$. \qedhere
    \end{enumerate} 
\end{proof}

\begin{corollary}
    The only balanced sign semidirect products are $D_6$, $D_{10}$, $Q_{12}$, and $Q_{20}$.
\end{corollary}

\begin{proof}
    By \autoref{prop:sign-semidirect-Q-C3-or-C5} and 
    \autoref{prop:sign-semidirect-P-cyclic}, all 
    balanced sign semidirect products are of the form 
    \(C_{2^n} \ltimes Q\), where \(Q\in \ps{C_3,C_5}\). 
    \begin{enumerate}
        \item If \(n=1\), we get \(D_6\) and \(D_{10}\) for \(Q=C_3\) and \(Q=C_5\), respectively. 
        \item If \(n=2\), we get \(Q_{12}\) and \(Q_{20}\) for \(Q=C_3\) and \(Q=C_5\), respectively.
        \item If \(n\ge 3\), then the index $8$ subgroup of $C_{2^n}$ is contained in $K$, so there exists a quotient of $C_{2^n} \ltimes Q$ isomorphic to \(C_8\ltimes Q\). However, this is imbalanced by \autoref{lem:sign-semidirect-imbalanced-C8:Cp}, so $C_{2^n} \ltimes Q$ is imbalanced by \autoref{lem:subquotients}.  \qedhere
    \end{enumerate}
\end{proof}

\subsection{Non-abelian 2-groups}
Let $P$ be a non-abelian $2$-group.

\begin{lemma}\label{lem:balanced-2-groups-of-small-order}
    The non-abelian balanced $2$-groups of orders $\leq 32$ are precisely:
    \begin{enumerate}
        \item Order 8: $Q_8$ and $D_8$.
        \item Order 16: $Q_8 \times C_2$, $Q_{16}$, and $C_4 \ltimes C_4$.
        \item Order 32: $Q_8 \times (C_2)^2$.
    \end{enumerate}
\end{lemma}

\begin{proof}
    This follows from a computation in Sage, which we defer to 
    \autoref{appendix:groups of order 8 16 and 32 computation}. 
\end{proof}

\begin{lemma}\label{D8-subquotient}
    Any non-Dedekind balanced $2$-group $P$ has a subquotient equal to $D_8$.
\end{lemma}

\begin{proof}
    Let $P$ be a non-Dedekind balanced $2$-group. Let $H$ be a minimal non-normal subgroup of $P$ and pick $g \not\in N_G(H)$. By \autoref{prop:balanced-equivalent-coset-condition}, $K = H \cap gHg^{-1}$ is an index $2$ subgroup of $H$, and by the minimality of $H$, $K$ is normal in $P$. By \autoref{lem:subquotients}, the quotient $\tilde{P} = P/K$ is balanced with non-normal subgroup $\tilde{H} = H/K$ of order $2$. It suffices to show that $D_8$ is a subgroup of $\tilde{P}$, so we may replace $P$ and $H$ with $\tilde{P}$ and $\tilde{H}$, respectively, and assume $H = \langle h \rangle$ has order $2$.

    As before, with $g \not\in N_G(H)$, we have one of $g^2 = 1$ or $(gh)^2 = 1$ by \autoref{lem:squares in H with probability 1/2}. In either case, that the subgroup generated by $g$ and $h$ is a non-abelian dihedral group. By \autoref{lem:dihedral-imbalanced} and \autoref{lem:subquotients}, this subgroup must be $D_8$.
\end{proof}

\begin{proposition}
    There are no non-Dedekind balanced $2$-groups of order $\geq 32$.
\end{proposition}

\begin{proof}
    Assume for the sake of contradiction that $P$ is a non-Dedekind balanced $2$-group of order $\ge 32$. By \autoref{D8-subquotient}, $P$ has a subquotient equal to $D_8$. On the other hand, since $P$ is a $2$-group, for any (normal) subgroup $H$ of $P$, there exists a chain of (normal) subgroups $H < H_1 < \cdots < H_k < P$, where each consecutive inclusion is of index $2$. Consequently, there exists a subquotient $\tilde{P}$ of $P$ of order $32$ with $D_8$ as a further subquotient. Since $\tilde{P}$ must be non-abelian, by \autoref{lem:balanced-2-groups-of-small-order}, $\tilde{P} = Q_8 \times (C_2)^2$. However, this is impossible since $\tilde{P}$ is Dedekind, so its subquotients must also be Dedekind, whereas $D_8$ is not Dedekind.
\end{proof}

\begin{corollary}
    The only balanced non-abelian $2$-groups are $D_8$, $Q_{16}$, $C_4 \ltimes C_4$, and $Q_8 \times (C_2)^n$, where $n \ge 0$.
\end{corollary}

This completes the proof of the classification of balanced groups in \autoref{thm:classification-of-balanced-groups}.

\section{Further Questions}\label{sec:Further-Questions}

It is vacuously true that all Dedekind groups are weakly balanced. \autoref{cor:three cases of weakly balanced groups} shows that any non-abelian weakly balanced group $G$ must fall under one of three cases. We have also shown in \autoref{lem:Direct products are weakly balanced iff Hamiltonian} that weakly balanced groups of type $(1)$ must be Hamiltonian, so their classification is given by \autoref{thm:baer-structure}. This leaves the weakly balanced groups of types $(2)$ and $(3)$. In both these cases, there exist groups which are neither Dedekind nor balanced.

\begin{enumerate}
    \item Type $(2)$: $D_{2n} = C_2 \ltimes C_n$ is weakly balanced for all $n$. More generally, the sign semidirect product $C_2 \ltimes Q$ is weakly balanced for all odd abelian $Q$. The sign semidirect products $C_8 \ltimes C_3$ and $Q_8 \ltimes C_3$ are also weakly balanced.
    \item Type $(3)$: The three non-trivial semidirect products $C_4 \ltimes C_8$ are all weakly balanced. The group $QD_{32}$ is also weakly balanced.
\end{enumerate}

\begin{question}
    Classify all weakly balanced groups.
\end{question}

One can also generalize the question to infinite groups. Define an \textit{infinite} group $G$ to be \textbf{balanced} if $|AA^{-1}| = |A^{-1}A|$ for all \textit{finite} subsets $A \subset G$. For example, the infinite group $Q_8 \times (C_2)^{\infty}$ is balanced by the same argument as in \autoref{prop:Hamil-2-nian}.

\begin{question}
    Classify all balanced infinite groups.
\end{question}

\appendix
\section{Proof of \autoref{lem:balanced-2-groups-of-small-order}}
\label{appendix:groups of order 8 16 and 32 computation}
We prove \autoref{lem:balanced-2-groups-of-small-order} by exhaustion using SageMath Version 10.7 with the GAP interface \cite{Sage, GAP4}. All code is available 
at 
\begin{center}
    \url{https://github.com/pr4-kp/Balanced-Groups-Computations}
\end{center}

For non-abelian groups \(G\) of order \(8\) and \(16\), we can run a brute 
force computation on all subsets of \(G\) to see if they are balanced or 
imbalanced, which gives us parts (1) and (2). Note that we can make 
this computation slightly faster by noting that if 
\(|A| > |G|/2\), then \(|AA\inv| = |A\inv A|\) (see \cite{tao12}). 

For non-abelian groups \(G\) of order \(32\), we first eliminate all 
groups that have a subgroup isomorphic to an imbalanced group, since
they are then imbalanced by \autoref{lem:subquotients}. This leaves us 
with the GAP groups of order \(32\) and ID in: 
\[\ps{2, 4, 5, 10,12, 13, 14, 17, 20, 23, 26, 32, 35, 41, 47}\]
Importantly, 
the ID 47 corresponds to \(Q_8 \times (C_2)^2\), which we know is 
balanced by \autoref{prop:Hamil-2-nian}. By sampling random sets from the other groups, we find 
that none of them are balanced except for \(Q_8 \times (C_2)^2\).
This completes part (3). 

\printbibliography

@misc{tao,
  title         = {Product set estimates for non-commutative groups},
  author        = {Terence Tao},
  year          = {2011},
  eprint        = {math/0601431},
  archiveprefix = {arXiv},
  primaryclass  = {math.CO},
  url           = {https://arxiv.org/abs/math/0601431}
}

@inbook{baer,
  url         = {https://doi.org/10.1515/9783111561769-004},
  title       = {Situation der Untergruppen und Struktur der Gruppe},
  booktitle   = {Acht Arbeiten Alfred Loewy zum 60. Geburtstag am 20. Juni 1933 gewidmet},
  author      = {Reinhold Baer},
  editor      = {Lothar Heffter},
  publisher   = {De Gruyter},
  address     = {Berlin, Boston},
  pages       = {12--17},
  doi         = {doi:10.1515/9783111561769-004},
  isbn        = {9783111561769},
  year        = {1933},
  lastchecked = {2025-09-02}
}

@article{zhao2010counting,
  title     = {Counting {MSTD} sets in finite abelian groups},
  author    = {Zhao, Yufei},
  journal   = {Journal of Number Theory},
  volume    = {130},
  number    = {10},
  pages     = {2308--2322},
  year      = {2010},
  publisher = {Elsevier}
}

@article{martin2006setssumsdifferences,
  title   = {Many sets have more sums than differences},
  author  = {Martin, Greg and O'Bryant, Kevin},
  journal = {arXiv preprint math/0608131},
  year    = {2006}
}

@article{kim2020constructionsgeneralizedmstdsets,
  title     = {Constructions of generalized MSTD sets in higher dimensions},
  author    = {Kim, Elena and Miller, Steven J},
  journal   = {Journal of Number Theory},
  volume    = {235},
  pages     = {358--381},
  year      = {2022},
  publisher = {Elsevier}
}

@misc{ascoli2022sum,
  title         = {Sum and Difference Sets in Generalized Dihedral Groups},
  author        = {Ruben Ascoli and Justin Cheigh and Guilherme Zeus Dantas e Moura and Ryan Jeong and Andrew Keisling and Astrid Lilly and Steven J. Miller and Prakod Ngamlamai and Matthew Phang},
  year          = {2022},
  eprint        = {2210.00669},
  archiveprefix = {arXiv},
  primaryclass  = {math.NT},
  url           = {https://arxiv.org/abs/2210.00669}
}

@online{tao12,
  author       = {Terence Tao},
  title        = {254B, Notes 5: Product theorems, pivot arguments, and the Larsen--Pink non-concentration inequality},
  date         = {2012-02-05},
  url          = {https://terrytao.wordpress.com/2012/02/05/254b-notes-5-product-theorems-pivot-arguments-and-the-larsen-pink-non-concentration-inequality/},
  organization = {What's new}
}

@manual{Sage,
  shorthand = {Sage},
  key       = {SageMath},
  author    = {{The Sage Developers}},
  title     = {{S}ageMath, the {S}age {M}athematics {S}oftware {S}ystem ({V}ersion 10.6)},
  note      = {{\url{https://www.sagemath.org}}},
  year      = {2025}
}

@manual{GAP4,
  shorthand    = {GAP},
  organization = {The GAP~Group},
  title        = {{GAP -- Groups, Algorithms, and Programming,
                  Version 4.14.0}},
  year         = 2024,
  url          = {\url{https://www.gap-system.org}}
}

@misc{duvivier2025,
  title         = {Comparing Left and Right Quotient Sets in Groups},
  author        = {Julian Duvivier and Xiaoyao Huang and Ava Kennon and Say-Yeon Kwon and Steven J. Miller and Arman Rysmakhanov and Pramana Saldin and Ren Watson},
  year          = {2025},
  eprint        = {2509.00611},
  archiveprefix = {arXiv},
  primaryclass  = {math.NT},
  url           = {https://arxiv.org/abs/2509.00611}
}

@article{Herzog-Kaplan-Longobardi-Maj,
  author     = {Herzog, Marcel and Kaplan, Gil and Longobardi, Patrizia and
                Maj, Mercede},
  title      = {Products of subsets of groups by their inverses},
  journal    = {Beitr. Algebra Geom.},
  fjournal   = {Beitr\"age zur Algebra und Geometrie. Contributions to Algebra
                and Geometry},
  volume     = {55},
  year       = {2014},
  number     = {2},
  pages      = {311--346},
  issn       = {0138-4821,2191-0383},
  mrclass    = {20F99 (20F50)},
  mrnumber   = {3263247},
  mrreviewer = {Arye\ Juh\'asz},
  doi        = {10.1007/s13366-013-0141-y},
  url        = {https://doi.org/10.1007/s13366-013-0141-y}
}

\end{document}